\documentclass[preprint]{elsarticle}
\usepackage{enumerate}
\usepackage{enumitem}
\usepackage{graphicx,color}
\usepackage{amssymb}
\usepackage{amssymb}
\usepackage{amsmath}
\usepackage{appendix}
\usepackage{epstopdf}
\usepackage{color}
\usepackage{natbib}
 \usepackage{lineno}

\newtheorem{theorem}{Theorem}[section]
\newtheorem{corollary}[theorem]{Corollary}
\newtheorem{lemma}[theorem]{Lemma}
\newtheorem{remark}[theorem]{Remark}

\newtheorem{proposition}[theorem]{Proposition}

\newenvironment{proof}[1][Proof]{\noindent\textbf{#1.} }{\ \rule{0.5em}{0.5em}}

\journal{Journal of Differential Equations}

\begin{document}

\begin{frontmatter}

\title{Local existence of solutions to the initial-value problem for one-dimensional strain-limiting viscoelasticity }

\author[1]{H. A. Erbay }
\ead{husnuata.erbay@ozyegin.edu.tr}

\author[2]{A. Erkip}
\ead{albert@sabanciuniv.edu}

\author[2]{Y. \c Seng\"{u}l \corref{cor1}}
\ead{yaseminsengul@sabanciuniv.edu}

\address[1]{Department of Natural and Mathematical Sciences, Faculty of Engineering, Ozyegin University,  Cekmekoy 34794, Istanbul, Turkey}
%              Tel.: +90 (216) 564 9489\\
%              Fax: +90 (216) 564 9057}
\address[2]{Faculty of Engineering and Natural Sciences, Sabanci University, Tuzla 34956, Istanbul, Turkey}
%              Tel.: +90 (216) 483 9508\\
%              Fax: +90 (216) 483 90 05\\
%\address[3]{Faculty of Engineering and Natural Sciences, Sabanci University, Tuzla 34956, Istanbul, Turkey}
%              Tel.: +90 (216) 483 9541\\
%              Fax: +90 (216) 483 90 05\\
%\fntext[r1]{This  work  was  supported  by  the  Scientific  and  Technological Research Council of Turkey (T\"{U}B\.{I}TAK) under the grant 116F093.}
\cortext[cor1]{Corresponding author}

\title{Local existence of solutions to the initial-value problem for one-dimensional strain-limiting viscoelasticity}

\begin{abstract}
In this work we prove local existence of strong solutions to the initial-value problem arising in one-dimensional strain-limiting viscoelasticity, which is based on a nonlinear constitutive relation between the linearized strain, the rate of change of the linearized strain and the stress. The model is a generalization of the nonlinear Kelvin-Voigt viscoelastic solid under the assumption that the strain and the strain rate are small. We define an initial-value problem  for the stress variable and then, under the assumption that the nonlinear constitutive function is strictly increasing, we convert the problem to a new form for the sum of the strain and the strain rate. Using the theory of variable coefficient heat equation together with a fixed point argument we prove local existence of solutions. Finally, for several constitutive functions widely used in the literature we show that the assumption on which the proof of existence is based is not violated.
\end{abstract}

\begin{keyword}
viscoelasticity \sep strain-limiting theory \sep initial-value problem \sep local existence
\end{keyword}

\end{frontmatter}

%\linenumbers

\section{Introduction}\label{intro}
\setcounter{equation}{0}

Our main aim in the present work is to prove local well-posedness of the initial-value problem based on a model which explains nonlinear response of one-dimensional viscoelastic solids exhibiting limiting strain behaviour.

The most popular constitutive theories for modelling deformable elastic solids are explicit models for the stress, that is, an explicit expression is given for the stress in terms of the kinematical variables (such as the strain). Such models form a special sub-class of implicit models introduced by Rajagopal in \cite{Rajagopal-2003}, \cite{Rajagopal-2007}, where an implicit constitutive relation is specified between the stress and the strain. As a consequence of the linearization under the assumption that the gradient of the displacement is small, one can develop a model where a nonlinear relationship between the linearized strain and the stress is specified. This is  called a \textit{strain-limiting model} and, in fact, it is capable of explaining the bahaviour observed in some experiments (see e.g. \cite{Rajagopal-2014}, \cite{Saito-et-all}). 

There has been quite an interest in such strain-limiting models recently (see e.g. \cite{BuMaRaSu}, \cite{Bus-2009}, \cite{Bus-Raj-2011}, \cite{Erbay-Sengul}, \cite{Huang-Raj-Dai}, \cite{Mai-Walton-2015}, \cite{Raj-Sac-2014}) where most of the studies are for elastic materials.  For viscoelasticity one can refer to Muliana et al. \cite{Mul-Raj-Wine} who developed a quasi-linear strain-limiting viscoelastic model where the linearized strain is expressed as an integral of a nonlinear measure of the stress. Similarly, Barta \cite{Barta} considered viscoelastic response described by an integral model involving a nonlinear relation between the strain and the history of the stress, and proved local existence of solutions to the initial and boundary-value problem with Dirichlet boundary conditions. In \cite{Rajagopal-2009} Rajagopal introduced a general implicit model allowing both the elastic and dissipative response to describe the response of viscoelastic solids.  The important point to note here is that this model is a generalization of  the nonlinear Kelvin-Voigt viscoelastic solid model under the assumption that the strain and strain rate are small. By considering special  subclasses of the general implicit model of \cite{Rajagopal-2009}, the existence of a weak solution \cite{BuMaRaj}, the circularly polarized wave propagation  \cite{Raj-Sac-2014} and the existence of one-dimensional traveling waves \cite{Erbay-Sengul} have been studied. In this article, we are interested in local well-posedness of a one-dimensional initial-value problem derived using the constitutive equation \eqref{Raj-model} which is a special case of the model introduced in \cite{Rajagopal-2009}.

In the present work, we convert the equation for the stress to obtain an equation in a new variable defined as the sum of the strain and the strain rate, and write it as a time-dependent heat equation. We use the results related to the variable coefficient heat equation (cf. \cite{Kaw-Shi}, \cite{Kob-Pec-Shi}, \cite{Pecher}) and the techniques from the theory of elliptic operators. The proof of the main theorem includes linearization around a given state, definition of a contractive mapping and the usage of Banach's fixed point theorem. It is worth noting that the analysis which is done in $\mathbb{R}$ in this work could also be done for a bounded domain under suitable boundary conditions. 

The structure of the paper is as follows. In Section \ref{ivp} we state the  constitutive relation of one-dimensional strain-limiting viscoelasticity and define the initial-value problem. In Section \ref{loc-existence} we prove local-in-time existence of solutions to the initial-value problem defined by \eqref{NL-eqn}-\eqref{ic-eta} and then, in Section \ref{strain-lim}, we discuss the condition for the local existence of solutions associated with several strain-limiting models.

\section{The initial value problem}\label{ivp}
\setcounter{equation}{0}

In this section we introduce the one-dimensional model based on strain-limiting viscoelasticity, define an initial-value problem in terms of the stress variable  and then convert the initial-value problem to a more tractable initial value problem in terms of a new variable. Hereafter we use dimensionless variables and parameters.

\subsection{The strain-limiting model of viscoelasticity}

 Rajagopal \cite{Rajagopal-2003}, \cite{Rajagopal-2007} introduced a new class of elastic materials, which includes classical elastic materials as special cases. The  response of such elastic materials is given by an implicit constitutive relation between the stress and the deformation gradient. Furthermore, to model elastic materials with limiting strain, he considered that the elastic materials can only undergo small displacement gradients. This requirement places no restriction on the magnitude of the stress while the strain remains small. In one-dimensional case, the implicit constitutive equation  is of the form $\phi(\epsilon, S)=0$ where $S(x, t)$ is the Cauchy stress and $\epsilon(x,t)$ is the linearized strain. The relationship between the linearized strain  $\epsilon(x,t)$  and the displacement function $u(x,t)$ is given by $\epsilon=u_{x}$. A particular case of the above constitutive relation is the strain-limiting model $\epsilon=h(S)$, where $h$ is a nonlinear function \cite{Rajagopal-2003, Rajagopal-2007}. It is worth mentioning that, for the strain-limiting elastic materials, the displacement gradients and the strain remain bounded even if the stress tends to infinity.

 Rajagopal \cite{Rajagopal-2009} extended the above approach to rate-type viscoelastic materials and proposed an implicit constitutive equation which takes the form $\phi(\epsilon, \epsilon_{t}, S)=0$ in one-dimensional case, where  the term $\epsilon_{t}$ is the time rate of change of the linearized strain  and is called the strain rate.  Erbay and \c{S}eng\"ul \cite{Erbay-Sengul} studied the particular case of  the strain-limiting rate-type viscoelastic model given in \cite{Rajagopal-2009} and \cite{Raj-Sac-2014} as
\begin{equation}\label{Raj-model}
    \epsilon + \nu \epsilon_{t} = h(S),
\end{equation}
where $\nu>0$ is the viscosity constant, and $h$ is a nonlinear function with $h(0)=0$. Clearly, the model \eqref{Raj-model}, which we will use from now on, is based on the assumption that both the strain and the time rate of strain are small. In \cite{Erbay-Sengul}, differentiating the equation of linear momentum $u_{tt} = S_{x}$ with respect to $x$, rewriting it in terms of the strain as  $\epsilon_{tt} = S_{xx}$ and using \eqref{Raj-model} in this equation, Erbay and \c{S}eng\"ul obtained the third-order semilinear equation
\begin{equation}\label{PDE-T}
    S_{xx} + \nu \,S_{xxt} = h(S)_{tt}, 
\end{equation}
and studied traveling wave solutions under the assumption of two constant equilibrium states at infinity. Clearly $S \equiv  0$ is an equilibrium solution of \eqref{PDE-T}. Since $S$ involves  both the elastic and dissipative response  of the body, the expectation that the first and the second terms on the left of \eqref{PDE-T} represent elastic and dissipative effects, respectively, in our strain-limiting model is not valid. Contrary to the equations derived from the classical viscoelastic models, the novelty of this equation lies in the fact that  the inertia term involving the second-order time derivative is nonlinear. Furthermore, it is worth noting that the auxiliary conditions (such as the initial or boundary conditions) given in terms of the stress $S$ rather than the displacement $u$ are more relevant for the above model.

\subsection{The initial-value problem}

Here we consider the case in which the body is all of $\mathbb{R}$ and the cases involving boundaries will be discussed in a future work. Consider the Cauchy problem defined by \eqref{PDE-T} and the initial conditions
\begin{equation}\label{ic}
    S(x, 0 ) = S_{0}(x), \quad S_{t}(x, 0) = S_{1}(x),
\end{equation}
where $S_{0}$ and $S_{1}$ are chosen suitably. In this work our aim is to prove that,  for initial data in appropriate function spaces and appropriate forms of $h$ appearing in the constitutive relation, the Cauchy problem  \eqref{PDE-T}-\eqref{ic} is locally well-posed in time.

We now define a new variable $\omega(x, t)$ which is the sum of the strain and strain rate: $\omega = \epsilon + \nu \epsilon_{t}$. Furthermore, we  assume that the function $h(\cdot)$ in \eqref{Raj-model} is sufficiently smooth and strictly increasing, i.e., $h'(z) > 0$ for any $z \in \mathbb{R}$. Note that this condition naturally arises in the linear theory since  in that case $h(z) = z$ and  $h'(z) = 1 >0$. Therefore, this is a mechanically meaningful assumption. As a result we arrive at a rate-type viscoelastic model where $S$ is a nonlinear function of $\omega$. We can rewrite \eqref{Raj-model} in the form
\begin{equation}\label{inverse}
    S = g(\omega),
\end{equation}
where $g$ is the inverse function of $h$ satisfying $g(0)=0$. As a consequence of the invertibility we observe that $g$ is a sufficiently smooth function and $g'(z) > 0$ for any $z \in \mathbb{R}$.  In other words, the stress $S$ is a smooth function of $\omega$ and it is strictly increasing. With this modification of the dependent variable from $S$ to $\omega$,  \eqref{PDE-T} becomes
\begin{equation}\label{nonlinear-eqn}
    \omega_{tt} = g(\omega)_{xx} + \nu g(\omega)_{xxt}.
\end{equation}
Similarly, the initial conditions \eqref{ic}   take the following form
\begin{equation}\label{icw}
    \omega(x, 0 ) = \omega_{0}(x), \quad \omega_{t}(x, 0) = \omega_{1}(x)
\end{equation}
for $\omega$, where we define $\omega_{0}$ and $\omega_{1}$ as
\begin{equation}\label{icwd}
    \omega_{0}(x)=h(S_{0}(x)), \quad \omega_{1}(x) = h'(S_{0}(x))S_{1}(x).
\end{equation}
The nature of the solution of the initial-value problem \eqref{nonlinear-eqn}-\eqref{icw} will depend on the outcome of the competition between the dissipative response of the material, and the wave steepening resulting from the nonlinear elastic response. It is worth mentioning that since $\omega = \epsilon + \nu \epsilon_{t}$, each of the terms on the right-hand side of \eqref{nonlinear-eqn} involves elastic and dissipative effects. Since $\omega$ involves both the strain and the strain-rate, again the first and second terms on the right-hand side of \eqref{nonlinear-eqn} do not represent purely elastic and purely dissipative effects, respectively.

Equation \eqref{nonlinear-eqn} can be converted into a system of equations 
\begin{eqnarray}
    && \omega_{t} - \phi_{x}  = 0  \label{omega-phi-1} \\
    &&\phi_{t} - \sigma(\omega, \omega_{t})_{x} = 0,  \label{omega-phi-2}
\end{eqnarray}
where  $\sigma = g(\omega) + \nu g(\omega)_{t}$ (i.e., $\sigma = S + \nu S_{t}$). Notice that $\omega \to u_{x}$, $\phi \rightarrow u_{t}$ and $\sigma \to S$ as $\nu \to 0^{+}$. Since we assume that $g'(\omega) > 0$, the system \eqref{omega-phi-1}-\eqref{omega-phi-2} for $\omega$ and $\phi$ is a quasilinear hyperbolic system with dissipation.  Equation \eqref{omega-phi-1} imposes that $\omega$ and $\phi$ derive from a ``potential" function so that $\omega=\eta_{x}$ and $\phi=\eta_{t}$. So, in order to write \eqref{nonlinear-eqn} in the divergence form we use the potential function $ \eta(x, t) = \int_{-\infty}^{x} \omega(y, t) dy$. Furthermore we assume that the initial data $\omega_{0}$ and $\omega_{1}$ for \eqref{nonlinear-eqn} are in the form  $\omega_{0}=(\eta_{0})_{x}$ and $\omega_{1}=(\eta_{1})_{x}$. Substituting these expressions into \eqref{nonlinear-eqn} and \eqref{icw} and cancelling out one space derivative, we get the Cauchy problem
\begin{eqnarray}
    && \eta_{tt} = g(\eta_{x})_{x} + \nu g(\eta_{x})_{xt}, \qquad x \in \mathbb{R}, \quad  t > 0  \label{NL-eqn} \\
    && \eta(x, 0 ) = \eta_{0}(x), \quad \eta_{t}(x, 0) = \eta_{1}(x),\qquad x \in \mathbb{R}.  \label{ic-eta}
\end{eqnarray}
This formal derivation holds if the terms tend to zero as $x \rightarrow \pm\infty$. This is assured when we show later that $\eta$ and $\eta_{t}$ belong to appropriate Sobolev spaces. We are interested in smooth solutions to \eqref{NL-eqn}-\eqref{ic-eta} corresponding to smooth initial data. The analysis in the next section will be based on the assumptions $\eta_{0}(\pm \infty) =0$ and $\eta_{1}(\pm \infty) = 0$. See Appendix for the mechanical interpretation and implication of these conditions.

We use the standard notations for Lebesgue and Sobolev spaces as well as the spaces of continuous functions.  For the rest of the paper, $H^{s} = H^{s}(\mathbb{R})$ denotes the $L^{2}$-based Sobolev space of order $s$ on $\mathbb{R}$ with the usual norm $\| z \|_{H^{s}} = \left(\int_{\mathbb{R}}(1+\xi^{2})^{s} |\widehat{z}(\xi)|^{2} d\xi\right)^{1/2} $, where the symbol \,$\widehat{}$\, denotes the Fourier transform. Similarly, we denote the norm in the space $L^{\infty}= L^{\infty}(\mathbb{R})$ as $\|z\|_{L^{\infty}} = \underset{x \in \mathbb{R}}{\mathrm{ess \,sup}}\,\, | z(x) |$. With $C^{k}(\mathbb{R})$, $k = 0, 1, 2, \ldots$ we represent the set of functions on $\mathbb{R}$ that are $k$-times continuously differentiable. Also $C([0, T]; H^{s})$ denotes the space of all $H^{s}$-valued functions $z$ on the interval $[0, T]$ of real numbers such that $z$ is strongly continuous on $[0, T]$. Similarly, $C^{1}([0, T]; H^{s})$ denotes the space of all continuously differentiable $H^{s}$-valued functions. Finally, $C$ denotes a generic positive constant.

\section{Local existence}\label{loc-existence}
\setcounter{equation}{0}

It is well-known that the equations of nonlinear elastodynamics do not admit, in general, global-in-time smooth solutions even for smooth initial data. Moreover, smooth solutions usually exist on a finite time interval only. However, in nonlinear viscoelasticity, the dissipation induced by the viscoelastic effect makes the problem more tractable and allows one to show that local-in-time smooth solutions exist over a longer time interval than that in the elastic case, or that, under certain assumptions, smooth solutions exist globally in time. The aim in this work, as a first step for the qualitative analysis, is to investigate local-in-time existence of solutions for the equations \eqref{NL-eqn}-\eqref{ic-eta} governing the motion of one-dimensional strain-limiting viscoelastic media. We, first, put \eqref{NL-eqn} in a form for which we will be able to define a nonlinear operator, and then use some techniques from elliptic operator theory \cite{Pecher}.

We now visualize \eqref{NL-eqn} as a parabolic type equation for $\eta_{t}$ with source term $g(\eta_{x})_{x}$. To get a bounded operator we add $\eta_{t}$ to both sides of \eqref{NL-eqn}. Explicitly writing the time derivative of the nonlinear term on the right-hand side of \eqref{NL-eqn} we get
\begin{equation}\label{eta-eqn}
\eta_{tt}  +\eta_{t} - \nu \Big(g'(\eta_{x}) \eta_{xt}\Big)_{x} = g(\eta_{x})_{x} + \eta_{t}.
\end{equation}
Now, letting $\eta_{t} = \phi$ we obtain the system
\begin{eqnarray}
&& \eta_{t} = \phi, \label{first-A-sys} \\
&& \phi_{t} + A(t) \phi = G(t) + F(t),  \label{second-A-sys} 
\end{eqnarray}
where 
\begin{equation}\label{A-F}
A(t) = 1 - \nu D_{x}(g'(\eta_{x}) D_{x}), \quad G(t) = g(\eta_{x})_{x}, \quad F(t) =\eta_{t}.
\end{equation}
Note that, we have divided the right-hand side of \eqref{eta-eqn} into two parts as $G(t)$ and $F(t)$ in \eqref{second-A-sys} since we will need to obtain estimates in different spaces.

\subsection{Properties of the operator $A(t)$}

In general, denoting $g'(\eta_{x}) = a(x, t),$ the second-order differential operator $A(t)$ can be written in divergence form as
\begin{equation}\label{A}
A(t) \phi = - \nu (a(x, t)\phi_{x})_{x} + \phi.
\end{equation}
From now on we assume that $a(x, t) \in C^{1}([0, T]; H^{s-1})$ for some $T > 0$. 
Using the product estimate
\begin{equation}\label{product-estimate}
\|u v\|_{H^s}\leq C \left(\|u\|_{L^{\infty}} \| v\|_{H^{s}} + \|u\|_{H^{s}} \|v\|_{L^{\infty}}\right)
\end{equation}
for $u, v \in H^{s} \cap L^{\infty}$, and the Sobolev embedding, we can show that $\|A u\|_{H^{s-2}} \leq C \|u\|_{H^{s}}$ for $s > 5/2$.

For an interval $[0, T]$, the operator $A: H^{s} \to H^{s-2}$ will be uniformly elliptic in the spatial variable $x$ if there exists a constant $\bar{\theta} > 0$ such that $\nu a(x, t) \geq \bar{\theta}$ for all $x \in \mathbb{R}$ and $t \in [0, T]$.  Throughout the rest of the paper, we will assume that there exists a constant $\theta > 0$ such that
\begin{equation}\label{pos}
\underset{x \in \mathbb{R}, t \in [0, T]}{\inf} \,\,a(x, t) = \theta > 0 .
\end{equation}
Since $\nu > 0$, this assumption guarantees the uniform ellipticity of $A$. Also, it is important to note that we added the factor $1$ to the operator in \eqref{eta-eqn} so that we are able to invert $A(t)$. This makes it possible that all calculations done in a bounded domain also hold in $\mathbb{R}$.
Below we refer to a general result about regularity of the inverse of elliptic operators in Sobolev spaces. 

\begin{lemma}\label{A-inverse}
Let $s> 5/2$ and $T>0$. Assume that the operator $A(t)$ is given as \eqref{A} with $a(x, t) \in C^{1}([0, T]; H^{s-1})$.  Also assume that \eqref{pos} holds. Then there exists an inverse operator $A^{-1}(t)$ of $A(t)$ which satisfies the estimates
\begin{eqnarray}
&& \|A^{-1}(t) z \|_{H^{s}} \leq C \| z \|_{H^{s-2}}, \label{A-inv}\\
&& \|(A^{-1})_{t}(t) z \|_{H^{s}} \leq C \| z \|_{H^{s-2}}, \label{A-t-inv}
\end{eqnarray}
for all $t \in [0, T]$, where constant $C$ is independent of $t$.
\end{lemma}
\begin{proof}
See \cite[Sec. 27.3]{Wloka} for a bounded domain in the case of variable coefficient parabolic differential equation, and \cite[Ex. 28.8]{Wloka} for a way to extend it to $\mathbb{R}$, which is possible due to the factor $1$ in the definition of operator $A(t)$.
\end{proof}

\subsection{Local existence for variable coefficient heat equation}\label{Var-Coef-Heat}

We can view \eqref{second-A-sys} with $A(t)$ defined by \eqref{A} as a variable coefficient heat equation. Consider the abstract (inhomogeneous) initial-value problem
\begin{align}
& \phi_{t} + A(t) \phi = f(t), \qquad 0 \leq \tau \leq t \leq T,\label{evol-sys-1} \\
& \phi(\tau)=\phi_{0}. \label{evol-sys-2}
\end{align}
We know local existence of a unique solution for \eqref{evol-sys-1}-\eqref{evol-sys-2} from \cite[Chp. 5]{Pazy}  (see also \cite[Thm. 26.1]{Wloka}  and \cite[Chp. 20]{Hille}) for bounded domains. This result can be extended to $\mathbb{R}$ by using the fact that the operator $A(t)$ is an isomorphism. The existence of such a unique solution  provides us an evolution system $\Phi(t, \tau)$. Moreover, we know that every classical solution $\phi$ is given by
\[\phi(t) = \Phi(t, \tau) \phi_{0} + \int_{\tau}^{t} \Phi(t, r) f(r) dr.\]
The main properties of $\Phi(t, \tau)$ are given in the next result  (see \cite[Chp. 5]{Pazy}).
\begin{lemma}\label{Pazy-thm}
For every $0 \leq \tau \leq t \leq T$, $\Phi(t, \tau)$ is a bounded linear operator on $H^{s}$ and
\begin{equation*}
\begin{array}{lllll}
\mathrm{(i)}\,\,\Phi(t, t) = \text{Id},\qquad \Phi(t, \tau) = \Phi(t, r) \Phi(r, \tau)\quad \text{for}\quad 0 \leq \tau \leq r \leq  t \leq T, \\
\mathrm{(ii)}\,\,(t, \tau) \to \Phi(t, \tau)\,\,\text{is strongly continuous for} \,\,\,0 \leq \tau \leq t \leq T,\\
\mathrm{(iii)}\,\, \Phi_{t}(t, \tau)= A(t) \Phi(t, \tau) \quad \text{for}\quad 0 \leq \tau \leq t \leq T \\
\mathrm{(iv)}\,\,\Phi_{\tau}(t, \tau) = -\Phi(t, \tau) A(\tau) \quad \text{for}\quad 0 \leq \tau \leq t \leq T.
\end{array}
\end{equation*}
\end{lemma}
Using these tools we now prove our main estimate for the solution of the initial-value problem defined by \eqref{second-A-sys} and $\phi(0) = \phi_{0}$.
\begin{proposition}\label{exist-heat}
Let $s > 5/2$, $T > 0$. Assume that $A$ is defined by \eqref{A} and that \eqref{pos} holds with $a(x, t) \in C^{1}([0, T]; H^{s-1})$.  Also assume that  $G \in  C^{1}([0, T]; H^{s-2})$  and $F \in  C([0, T]; H^{s})$. Then, equation \eqref{second-A-sys} with $\phi(0) = \phi_{0} \in H^{s}$ has a unique solution $\phi \in C([0, T]; H^{s})$  satisfying the estimate
\begin{equation}\label{u-est}
\begin{split}
 \|\phi(t)\|_{H^{s}} & \leq C\bigg( \|\phi_{0}\|_{H^{s}} +  \|G(0)\|_{H^{s-2}}  \\
 & \qquad  + \int_{0}^{t} \Big(\| G(\tau) \|_{H^{s-2}} + \|G_{\tau}(\tau) \|_{H^{s-2}} +  \| F(\tau)\|_{H^{s}}\Big) d\tau  \bigg) 
\end{split}
\end{equation}
for $0 \leq t \leq T$, where $C$ is a constant independent of $t$.
\end{proposition}
\begin{proof}
We now know from the above discussion that the solution of the initial-value problem for \eqref{second-A-sys} is given by
\begin{equation}\label{Duhamel-soln}
\phi(t) = \Phi(t, 0) \phi_{0} + \int_{0}^{t} \Phi(t, \tau) G(\tau) d\tau + \int_{0}^{t} \Phi(t, \tau) F(\tau) d\tau,
\end{equation}
where the properties of the evolution system $\Phi(t, \tau)$ are given in Lemma \ref{Pazy-thm}. 
From \eqref{Duhamel-soln} we have
\begin{equation*}
 \|\phi(t)\|_{H^{s}}  \leq C \|\phi_{0}\|_{H^{s}} +\left\| \int_{0}^{t} \Phi(t, \tau) G(\tau) d\tau \right\|_{H^{s}} +\left\| \int_{0}^{t} \Phi(t, \tau) F(\tau) d\tau \right\|_{H^{s}}
\end{equation*}
due to the fact that $\Phi(t, \tau)$ is a bounded operator.  For the last term on the right-hand side we have 
\begin{equation}\label{est-F}
\left\| \int_{0}^{t} \Phi(t, \tau) F(\tau) d\tau \right\|_{H^{s}} \leq C  \int_{0}^{t}\| F(\tau)\|_{H^{s}} d\tau.
\end{equation}
For the second term, by part (iv) of Lemma \ref{Pazy-thm}, we obtain 
\[\left\|\int_{0}^{t} \Phi(t, \tau) A(\tau) A^{-1}(\tau) G(\tau) d\tau\right\|_{H^{s}}  =  \left\|\int_{0}^{t}  \Big(- \partial_{\tau}\Phi(t, \tau)\Big) \Big(A^{-1}(\tau) G(\tau)\Big) d\tau\right\|_{H^{s}} .\]
Now, integration by parts and boundedness of $\Phi$ give
\begin{eqnarray}
&& \left\|\int_{0}^{t}  \Big(- \partial_{\tau}\Phi(t, \tau)\Big) \Big(A^{-1}(\tau) G(\tau)\Big) d\tau\right\|_{H^{s}} \nonumber \\
& & \qquad \qquad = \left\|  \int_{0}^{t} \Phi(t, \tau) \partial_{\tau}\Big(A^{-1}(\tau) G(\tau)\Big)d\tau - \Big(\Phi(t, \tau) A^{-1}(\tau) G(\tau)\Big)\Big|_{\tau=0}^{\tau=t} \right\|_{H^{s}} \nonumber \\
& &\qquad \qquad  \leq  C \int_{0}^{t} \left\| (A^{-1})_{\tau}(\tau) G(\tau) + A^{-1}(\tau) G_{\tau}(\tau) \right\|_{H^{s}} d\tau  \nonumber  \\
& &\qquad  \qquad \qquad \qquad \qquad + \big\|\Phi(t, t) A^{-1}(t) G(t) - \Phi(t, 0) A^{-1}(0) G(0)\big\|_{H^{s}} .\label{est-UG}
\end{eqnarray}
For the last term, we can add and subtract $\Phi(t, t) A^{-1}(0)G(0)$ and use boundedness of $\Phi$ again to write
\begin{eqnarray}
&&\!\!\!\!\!\!\!\!\!\!\!\!\!\!\!\!\!\!\!\!\!\!\!\!\!\!\big\|\Phi(t, t) A^{-1}(t) G(t) - \Phi(t, 0) A^{-1}(0) G(0)\big\|_{H^{s}}  \nonumber \\
&&\!\!\!\!\!\!\!\!\!\!\!\!\!\!\!\!\!\!\!\!\!\!\!\!\!= \big\|\Phi(t, t) A^{-1}(t) G(t) - \Phi(t, t) A^{-1}(0)G(0) \nonumber  \\
&& \qquad \qquad \qquad  \quad  +  \Phi(t, t) A^{-1}(0)G(0)  - \Phi(t, 0) A^{-1}(0) G(0)\big\|_{H^{s}}  \nonumber \\
&&\!\!\!\!\!\!\!\!\!\!\!\!\!\!\!\!\!\!\!\!\!\!\!\!\!\leq C \big\| A^{-1}(t)G(t)-A^{-1}(0) G(0)\big\|_{H^{s}} \nonumber  \\
&& \qquad \qquad \qquad  \quad  + \big\|\big(\Phi(t, t) - \Phi(t, 0)\big)A^{-1}(0) G(0) \big\|_{H^{s}}. \label{est-UG-2}
\end{eqnarray}
Also we have
\begin{eqnarray}
&& \!\!\!\!\!\!\!\!\!\!\!\!\!\!\!\!\!\!\!\!\!\!\!\!\!\!\!\!\!\!\!\! \big\| A^{-1}(t)G(t) - A^{-1}(0) G(0)\big\|_{H^{s}}   =\left\| \int_{0}^{t} \frac{\partial}{\partial \tau} \big(A^{-1}(\tau)G(\tau)\big) d\tau \right\|_{H^{s}} \nonumber \\
&& \qquad \qquad \qquad \leq   \int_{0}^{t} \big\| (A^{-1})_{\tau}(\tau) G(\tau) + A^{-1}(\tau) G_{\tau}(\tau)\big\|_{H^{s}}  d\tau, \label{est-UG-3}
\end{eqnarray}
and
\begin{equation}\label{est-UG-4}
\big\|\big(\Phi(t, t) - \Phi(t, 0)\big)A^{-1}(0) G(0) \big\|_{H^{s}} \leq  C \big\|A^{-1}(0) G(0)\big\|_{H^{s}}.
\end{equation}
Combining estimates \eqref{est-F}, \eqref{est-UG}, \eqref{est-UG-2}, \eqref{est-UG-3} and \eqref{est-UG-4} we obtain
\begin{equation*}
\begin{split}
 \|\phi(t)\|_{H^{s}} & \leq C \|\phi_{0}\|_{H^{s}} + C  \int_{0}^{t} \Big(\big\|(A^{-1})_{\tau}(\tau) G(\tau) \big\|_{H^{s}}+  \big\|A^{-1}(\tau) G_{\tau}(\tau) \big\|_{H^{s}}\Big) d\tau \\
 & \qquad  \qquad + C \big\|A^{-1}(0) G(0)\big\|_{H^{s}} + C \int_{0}^{t}\big\| F(\tau)\big\|_{H^{s}} d\tau.
\end{split}
\end{equation*}
Applying \eqref{A-inv} and \eqref{A-t-inv} we get \eqref{u-est} as required.
\end{proof}

Using \eqref{first-A-sys} and \eqref{u-est} we can now prove estimates for $\eta$ and $\eta_{t}$. 

\begin{corollary}\label{cor-heat}
Let $s > 5/2 $, $T > 0$. Assume that $A$ is defined by \eqref{A} and that \eqref{pos} holds with $a(x, t) \in C^{1}([0, T]; H^{s-1})$.  Also assume that  $G \in  C^{1}([0, T]; H^{s-2})$  and $F \in  C([0, T]; H^{s})$. Then there exists a solution $(\phi, \eta) \in \big(C([0, T]; H^{s}), C^{1}([0, T]; H^{s})\big)$ to \eqref{first-A-sys}-\eqref{second-A-sys} with $\eta(0)=\eta_{0}$ and $\phi(0)=\phi_{0}=\eta_{1}$ and satisfying the following estimates
\begin{align}
 \|\eta(t)\|_{H^{s}} & \leq \|\eta_{0}\|_{H^{s}} + C T\Big(\|\eta_{1}\|_{H^{s}} + \|G(0)\|_{H^{s-2}} \Big) \nonumber\\
 & \qquad + C\,T \,\int_{0}^{t} \Big(\|G(\tau)\|_{H^{s-2}} + \left\| G_{\tau}(\tau)\right\|_{H^{s-2}} + \|F(\tau)\|_{H^{s}}\Big) d\tau \label{eta-est}\\
 \|\eta_{t}(t)\|_{H^{s}} & \leq C \bigg( \|\eta_{1}\|_{H^{s}} + \|G(0)\|_{H^{s-2}}  \nonumber \\
 & \qquad   + \int_{0}^{t} \Big(\| G(\tau) \|_{H^{s-2}} + \|G_{\tau}(\tau) \|_{H^{s-2}} +  \| F(\tau)\|_{H^{s}}\Big) d\tau  \bigg) \label{eta-t-est}
 \end{align}
for $0 \leq t \leq T$, where $C$ is a constant.
\end{corollary}
\begin{proof}
Since
\begin{equation}\label{eta}
\eta(x, t)= \eta_{0}(x)  + \int_{0}^{t} \phi(x, \tau) d\tau,
\end{equation}
we have 
\[\|\eta(t)\|_{H^{s}} \leq \|\eta_{0}\|_{H^{s}} + \int_{0}^{t} \|\phi(\tau)\|_{H^{s}} d\tau.\]
From  \eqref{u-est} we immediately find \eqref{eta-est}. 
Recalling from \eqref{first-A-sys} that $\phi = \eta_{t}$ we obtain \eqref{eta-t-est} from \eqref{u-est}.
\end{proof}

\subsection{Assumptions about the nonlinearity}

We refer to \cite{Cons-Mol} for the following two lemmas about the nonlinearity.

\begin{lemma}\label{bound-lemma}
Let $g \in C^{\infty}(\mathbb{R})$ with $g(0)= 0$.  If $z \in H^{s}(\mathbb{R}) \cap L^{\infty}(\mathbb{R})$, $ s \geq 0,$ then
\[\|g(z) \|_{H^{s} \cap L^{\infty}} \leq K_{1} \|z\|_{H^{s} \cap L^{\infty}},\]
where $K_{1}$ depends only on $\|z \|_{L^{\infty}}$.
\end{lemma}

\begin{lemma}\label{llc-lemma}
Let $g \in C^{\infty}(\mathbb{R})$. If $z_{1}, z_{2} \in  H^{s}(\mathbb{R}) \cap L^{\infty}(\mathbb{R})$, $s \geq 0$, then
\[\|g(z_{1}) - g(z_{2}) \|_{H^{s} \cap L^{\infty}} \leq K_{2} \|z_{1} - z_{2}\|_{H^{s} \cap L^{\infty}},\]
where $K_{2}$ depends on $\|z_{1}\|_{H^{s} \cap L^{\infty}}$ and $\|z_{2}\|_{H^{s} \cap L^{\infty}}$.
\end{lemma}

We would like condition \eqref{pos} to hold with $a(x, t) = g'(\eta_{x})$. Since $g'(z)>0$ for all $z \in \mathbb{R}$, we have $g'(0) > 0$. Note that since stress increases with strain near equilibrium, this is a mechanically meaningful condition. Moreover, by continuity of  $g'(\cdot) $, we know that there must exist a constant $\delta > 0$ such that if $|z | \leq \delta$ then $g'(z) \geq \theta >0$. We will make use of this information to state the assumption on $\|\eta_{x}(t)\|_{L^{\infty}}$ below.

\subsection{Linearization} 

\noindent In this section we will use a linearization argument which is in principle the same as the localization technique used for local well-posedness of the initial-value problems  associated with viscoelasticity in many articles. For this, we consider \eqref{eta-eqn} and linearize around a given $v \in C^{1}([0, T]; H^{s})$ satisfying $v(x, 0) = \eta_{0}(x)$ and $v_{t}(x, 0) = \eta_{1}(x)$, to get
\begin{equation}\label{eta-eqn-linear}
\eta_{tt} + \eta_{t} - \nu (g'(v_{x}) \eta_{xt})_{x}= g(v_{x})_{x} + v_{t} .
\end{equation}
Letting $\eta_{t} = \phi$ again, we obtain the linear equation
\begin{equation}\label{A-v-sys}
\phi_{t} + A_{v}(t) \phi = G_{v}(t) + F_{v}(t),
\end{equation}
where 
\begin{equation}\label{A-G-F-v}
A_{v}(t) = 1 - \nu D_{x}(g'(v_{x}) D_{x}), \quad  G_{v}(t) = g(v_{x})_{x}, \quad F_{v}(t) =   v_{t}.
\end{equation}
Now, we can prove the following result.

\begin{proposition}\label{linear}
Let $s > 5/2$. Assume $g \in C^{r+1}$ with $r= [s]+1$. Let $v \in C^{1}([0, T]; H^{s})$ with $v(x, 0) = \eta_{0}(x)$ and $v_{t}(x, 0) = \eta_{1}(x)$, and $\|v_{x}(t)\|_{L^{\infty}}\leq \delta$ for all $t \in [0, T]$, be given. Then, equation \eqref{eta-eqn-linear} with initial data $\eta(x, 0) = \eta_{0} \in H^{s}$, $\eta_{t}(x, 0) = \eta_{1} \in H^{s}$, admits a unique solution $\eta \in C^{1}([0, T]; H^{s})$ satisfying the following estimates
 \begin{align}
  \|\eta(t)\|_{H^{s}}   & \leq  \|\eta_{0}\|_{H^{s}} +  C T \Bigg(K_{1}(\delta) \|\eta_{0}\|_{H^{s}} +  \|\eta_{1}\|_{H^{s}} \nonumber \\
& \qquad \qquad  \qquad +  \int_{0}^{t} \Big( K_{1}(\delta) \|v(\tau)\|_{H^{s}} + \big(K_{2}(\delta) \|v(\tau)\|_{H^{s}} \label{eta-est-v-1} \\
& \qquad \qquad \qquad +  \| g'(0)\|_{H^{s-1}} +1 \big) \|v_{\tau}(\tau)\|_{H^{s}} \Big) d\tau \Bigg), \nonumber  \\
 \|\eta_{t}(t)\|_{H^{s}} & \leq C \Bigg( \|\eta_{1}\|_{H^{s}} + K_{1}(\delta)  \|\eta_{0}\|_{H^{s}} + \int_{0}^{t} \Big( K_{1}(\delta) \|v(\tau)\|_{H^{s}} \nonumber \\
 & \qquad + \big(K_{2}(\delta) \|v(\tau)\|_{H^{s}} +  \| g'(0)\|_{H^{s-1}} + 1\big)\|v_{\tau}(\tau)\|_{H^{s}}  \Big) d\tau \Bigg), \label{eta-est-v-2} 
\end{align}
where $K_{1}(\delta)$ and $K_{2}(\delta)$ are the constants in  Lemma \ref{bound-lemma} and Lemma \ref{llc-lemma}, respectively.
\end{proposition}
\begin{proof}
We would like to use Corollary \ref{cor-heat} for which we need that, due to  \eqref{pos},  $g'(v_{x}(x, t)) \geq \theta > 0$ holds. We ensure this using $g'(0) > 0$ and the continuity argument mentioned above since $\|v_{x}(t)\|_{L^{\infty}}\leq \delta$ is assumed. Now, by Corollary \ref{cor-heat} we have existence of $\eta \in C^{1}([0, T]; H^{s})$ satisfying
 estimates \eqref{eta-est} and \eqref{eta-t-est}. Applying these to \eqref{A-v-sys} we obtain 
\begin{align}
\|\eta(t)\|_{H^{s}} & \leq \|\eta_{0}\|_{H^{s}} +C T\Big(\|\eta_{1}\|_{H^{s}} + \|G_{v}(0)\|_{H^{s-2}} \Big) \nonumber \\
&    + C\,T \,\int_{0}^{t} \Big(\|G_{v}(\tau)\|_{H^{s-2}} + \left\| (G_{v}(\tau))_{\tau}\right\|_{H^{s-2}} + \|F_{v}(\tau)\|_{H^{s}}\Big) d\tau, \label{eta-v-est} \\
 \|\eta_{t}(t)\|_{H^{s}} & \leq C \bigg( \|\eta_{1}\|_{H^{s}} +  \|G_{v}(0)\|_{H^{s-2}} \nonumber \\
 & \quad  + \int_{0}^{t} \Big(\| G_{v}(\tau) \|_{H^{s-2}} + \|(G_{v}(\tau))_{\tau} \|_{H^{s-2}} +  \| F_{v}(\tau)\|_{H^{s}}\Big) d\tau  \bigg).\label{eta-t-v-est}
 \end{align} 
Moreover, we have
 \begin{align}
\|G_{v}(t)\|_{H^{s-2}} & = \|g(v_{x}(t))_{x}\|_{H^{s-2}} = \|g(v_{x}(t))\|_{H^{s-1}} \nonumber\\
&  \leq K_{1}(\delta) \| v_{x}(t)\|_{H^{s-1}}  = K_{1}(\delta) \| v(t)\|_{H^{s}},\label{G-v-est}
 \end{align}
where we used Lemma \ref{bound-lemma} with constant $K_{1}(\delta)$. This immediately gives
\begin{equation}\label{G-v-0-est}
\|G_{v}(0)\|_{H^{s-2}} \leq K_{1}(\delta) \|v(0)\|_{H^{s}} = K_{1}(\delta)  \|\eta_{0}\|_{H^{s}}.
\end{equation}
Also, 
 \begin{equation*}
\|(G_{v})_{t}(t)\|_{H^{s-2}}  = \|g(v_{x})_{xt}(t)\|_{H^{s-2}} =  \|g(v_{x})_{t}(t)\|_{H^{s-1}} =   \| (g'(v_{x})v_{xt})(t)\|_{H^{s-1}}.
 \end{equation*}
Since $s > 5/2$ by the product estimate \eqref{product-estimate} we have 
\begin{equation*}
\begin{split}
&  \| (g'(v_{x}) v_{xt})(t)\|_{H^{s-1}}  \leq C \| g'(v_{x})(t)\|_{H^{s-1}} \|v_{xt}(t)\|_{H^{s-1}} \\
& \quad \leq C \Big(\| g'(v_{x})(t) - g'(0)\|_{H^{s-1}} \|v_{xt}(t)\|_{H^{s-1}} + \| g'(0)\|_{H^{s-1}} \|v_{xt}(t)\|_{H^{s-1}}\Big).
\end{split}
\end{equation*}
Hence, applying Lemma \ref{llc-lemma} to $g'$ we obtain
\begin{align}
\|(G_{v})_{t}(t)\|_{H^{s-2}} & \leq C \Big( K_{2}(\delta) \|v_{x}(t)\|_{H^{s-1}} \|v_{xt}(t)\|_{H^{s-1}} + \| g'(0)\|_{H^{s-1}} \|v_{xt}(t)\|_{H^{s-1}}\Big) \nonumber \\
&  =  C\|v_{t}(t)\|_{H^{s}}  \big(K_{2}(\delta) \|v(t)\|_{H^{s}} + \| g'(0)\|_{H^{s-1}}  \big).\label{G-v-t-est}
\end{align}
Combining \eqref{G-v-est},  \eqref{G-v-t-est} and using $\|F_{v} (t) \|_{H^{s}} = \|v_{t}(t) \|_{H^{s}}$  for the integral term in \eqref{eta-v-est} and \eqref{eta-t-v-est} we obtain 
\begin{equation*}
\begin{split}
 \int_{0}^{t} \Big(\|G_{v}(\tau)\|_{H^{s-2}} & + \left\| (G_{v}(\tau))_{\tau}\right\|_{H^{s-2}} + \|F_{v}(\tau)\|_{H^{s}}\Big) d\tau \\
& \leq \int_{0}^{t} \Big( K_{1}(\delta) \|v(\tau)\|_{H^{s}} + C \|v_{\tau}(\tau)\|_{H^{s}}\big(K_{2}(\delta) \|v(\tau)\|_{H^{s}} \\
& \qquad \qquad +   \| g'(0)\|_{H^{s-1}}\big) + \|v_{\tau}(\tau)\|_{H^{s}} \Big) d\tau.
\end{split}
\end{equation*}
Together with \eqref{G-v-0-est}, this implies  \eqref{eta-est-v-1} and \eqref{eta-est-v-2}  as required.
\end{proof}

\subsection{Local existence for the main problem}

We will now construct a fixed point scheme in order to prove local well-posedness for the Cauchy problem \eqref{NL-eqn}-\eqref{ic-eta}. To use Proposition \ref{linear}, we need that $\|\eta_{x}(t)\|_{L^{\infty}} \leq  \delta$ for all $t \in [0, T].$ Instead of assuming this directly, we will make the assumption that $\|\eta(t)\|_{H^{s}} \leq \bar{\delta}$ for all $t \in [0, T]$, which will imply the necessary condition since $s > 5/2$. This is because
\[\|\eta_{x}(t)\|_{L^{\infty}} \leq \bar{C} \|\eta_{x}(t)\|_{H^{s-2}} \leq \bar{C} \|\eta_{x}(t)\|_{H^{s-1}} \leq \bar{C} \| \eta(t)\|_{H^{s}} \leq \delta,\]
where $\bar{C}$ is the embedding constant and $\delta = \bar{C} \bar{\delta}$.

Another issue is related to boundedness of $\|\eta_{t}(t)\|_{H^{s}}$. We do not need that it is small but we need a uniform bound. Therefore we will make the assumption that $\|\eta_{t}(t)\|_{H^{s}} \leq M$ for all $t \in [0, T]$ where the constant $M$ will be chosen later. 

Now, we look for a solution $\eta$ in the Banach space defined by
\begin{equation*}
\begin{split}
X^{s}([0, T])  & = \Big\{z \in C^{1}([0, T]; H^{s})  \Big| z(0) = \eta_{0}, z_{t}(0) = \eta_{1},  \\
& \qquad \qquad \qquad \qquad \|z(t)\|_{H^{s}} \leq \bar{\delta}, \|z_{t}(t)\|_{H^{s}} \leq M, t \in [0, T] \Big\}
\end{split}
\end{equation*}
and endowed with the norm
\begin{equation}\label{norm-X}
\|z\|_{X^{s}([0, T])} = \underset{t \in [0, T]}{\sup} \Big(\| z(t) \|_{H^{s}} + \| z_{t}(t) \|_{H^{s}}\Big).
\end{equation}

\begin{theorem}\label{local-existence}
Let $s > 5/2$. Assume $g \in C^{r+1}$ with $r= [s]+1$. Assume also that $\eta_{0} \in H^{s}$ with $\| \eta_{0}\|_{H^{s}} \leq \frac{\bar{\delta}}{2 ( 1 + T_{0} K_{1}(\delta))}$ for some $T_{0} > 0$, where $\delta$ and $K_{1}(\delta)$ are as in Proposition \ref{linear}, and $\bar{\delta}$ is as in the definition of $X^{s}([0, T])$. Then, there exists a sufficiently small time $T > 0$ with $T \leq T_{0}$, such that the Cauchy problem \eqref{NL-eqn}-\eqref{ic-eta} admits a unique solution $\eta \in X^{s}([0, T])$.
\end{theorem}
\begin{proof}
%Let $g(z) = m(z) z$, define $k(z) = m'(z)z + m(z)$ and let $\eta_{t} = u$. 
Let $v \in X^{s}([0, T])$ be given. Define the operator
\begin{equation*}
\mathcal{K}(v) = \eta.
\end{equation*}
Clearly $\mathcal{K}$ is a map which takes $v$ to a solution of \eqref{eta-eqn-linear}. We need to show that for suitably chosen $T$, $\mathcal{K}$ has a unique fixed point in $X^{s}([0, T])$. First we show that $\mathcal{K}$ is a map from $X^{s}([0, T])$ into itself. Clearly, $\mathcal{K}(v(0)) = \eta_{0}$ and $\mathcal{K}(v_{t}(0)) = \eta_{1}$  as required. From the above discussion, we know that $\|v(t)\|_{H^{s}} \leq \bar{\delta}$ implies $\|v_{x}(t)\|_{L^{\infty}}\leq \delta$. Hence, by Proposition \ref{linear} we know that solution $\eta$ belongs to $X^{s}([0, T])$ and satisfies the estimates
\begin{equation*}
\begin{split}
 \|\eta(t)\|_{H^{s}}   & \leq  \|\eta_{0}\|_{H^{s}} +  C T \Bigg(K_{1}(\delta) \|\eta_{0}\|_{H^{s}} +  \|\eta_{1}\|_{H^{s}}  \\
& \qquad \qquad \qquad +  \int_{0}^{t} \Big( K_{1}(\delta) \|v(\tau)\|_{H^{s}} + \big(K_{2}(\delta) \|v(\tau)\|_{H^{s}} \\
& \qquad \qquad \qquad +   \| g'(0)\|_{H^{s-1}} +1 \big) \|v_{\tau}(\tau)\|_{H^{s}} \Big) d\tau \Bigg), \\
 \|\eta_{t}(t)\|_{H^{s}} & \leq C \Bigg( \|\eta_{1}\|_{H^{s}} + K_{1}(\delta)  \|\eta_{0}\|_{H^{s}} + \int_{0}^{t} \Big( K_{1}(\delta) \|v(\tau)\|_{H^{s}}  \\
 & \qquad + \big(K_{2}(\delta) \|v(\tau)\|_{H^{s}} +  \| g'(0)\|_{H^{s-1}} + 1\big)\|v_{\tau}(\tau)\|_{H^{s}}  \Big) d\tau \Bigg).
\end{split}
\end{equation*}
For the first estimate, since $\| \eta_{0}\|_{H^{s}} \leq \frac{\bar{\delta}}{2 ( 1 + T_{0} K_{1}(\delta))}$, we have
\[\|\eta(t)\|_{H^{s}} \leq \frac{\bar{\delta}}{2} + C T M+ CT^{2}\Big(K_{1}(\delta) \bar{\delta} + \big(K_{2}(\delta) \bar{\delta}+\| g'(0)\|_{H^{s-1}}+ 1\big) M\Big).\]
%By the assumption on the initial data, using continuity of $g'$ we know that $g'(\eta_{x}(x, t)) \geq \theta$ whenever $\|\eta_{x}(t)\|_{H^{s}} < \tilde{\delta}$ for $t \in [0, T]$, where $\theta$  is as in \eqref{pos} and $\tilde{\delta} > 0$ is sufficiently small. For $t=0$, taking $\tilde{\delta}   \leq \frac{\bar{\delta}}{ 2(1+ T K_{1}(\delta))}$, where $K_{1}(\delta)$ is the Lipschitz constant from Lemma \ref{bound-lemma} and $\bar{\delta}$ is as in $X^{s}([0, T])$, we obtain
Choosing $T$ small enough, we can ensure that $\|\eta(t)\|_{H^{s}} \leq \bar{\delta}$. For the second estimate, we have
\[\|\eta_{t}(t)\|_{H^{s}} \leq \frac{\bar{\delta}}{2} + C \|\eta_{1}\|_{H^{s}}+ C T \Big(K_{1}(\delta) \bar{\delta} + \big( K_{2}(\delta)\bar{\delta} + \| g'(0)\|_{H^{s-1}}+1\big)M \Big).\]
The last term with $T$ in front can be made small so that we obtain $\|\eta_{t}(t)\|_{H^{s}} \leq C \|\eta_{1}\|_{H^{s}} + \bar{\delta}$. Choosing $M \geq C \|\eta_{1}\|_{H^{s}} + \bar{\delta}$  we obtain $ \|\eta_{t}(t)\|_{H^{s}} \leq M$ so that $\eta \in X^{s}([0, T])$.

We now prove that for small enough $T$, the map $\mathcal{K}$ is a contractive in $X^{s}([0, T])$.
For this, let us take $v, \bar{v} \in X^{s}([0, T])$ such that $\eta = \mathcal{K}(v)$ and $\bar{\eta} = \mathcal{K}(\bar{v})$ with $\phi = \eta_{t}$, $\bar{\phi}= \bar{\eta}_{t}$. Set $V = v - \bar{v}$ and $\psi = \phi - \bar{\phi}$ so that $V(x, 0) = V_{t}(x, 0) = 0$ and $\psi(x, 0) = \psi_{t}(x, 0) = 0$. Then, writing \eqref{A-v-sys} for $\phi$ and $\bar{\phi}$ and subtracting we have
\begin{equation}\label{U-eqn}
\psi_{t} + A_{v}(t) \psi = (\tilde{G}_{v}(t) - \tilde{G}_{\bar{v}}(t)) + (F_{v}(t) - F_{\bar{v}}(t)),
\end{equation}
where $\tilde{G}_{v}(t) = G_{v}(t) - A_{v}(t) \bar{\phi}$ and $\tilde{G}_{\bar{v}}(t) = G_{\bar{v}}(t) - A_{\bar{v}}(t) \bar{\phi}$. 
By definition of the operator $A_{v}$ in \eqref{A-G-F-v}  and the product estimate \eqref{product-estimate} we have
\begin{equation*}
\begin{split}
\big\| (A_{v}(t) - A_{\bar{v}}(t)) \bar{\phi}(t) \big\|_{H^{s-2}} & = \nu \big\| \big((g'(\bar{v}_{x}) - g'(v_{x})) \bar{\phi}_{x}\big)_{x} (t) \big\|_{H^{s-2}} \\
& = \nu \big\|(g'(\bar{v}_{x}) - g'(v_{x})) \bar{\phi}_{x}(t) \big\|_{H^{s-1}}\\
& \leq \nu \big\|\bar{\phi}_{x}(t)\big\|_{H^{s-1}} \big\|\big(g'(\bar{v}_{x}) - g'(v_{x})\big)(t) \big\|_{H^{s-1}}
\end{split}
\end{equation*}
Now, by Lemma \ref{llc-lemma} we can continue as
\begin{equation*}
\begin{split}
&  \big\| (A_{v}(t) - A_{\bar{v}}(t)) \bar{\phi} \big\|_{H^{s-2}}   \leq C K_{2}(\delta)   \big\|\bar{v}_{x}(t) - v_{x}(t) \big\|_{H^{s-1}} \big\|\bar{\eta}_{xt}(t) \big\|_{H^{s-1}} \\
& \qquad   \leq  C   K_{2}(\delta) \big\|\bar{v}(t) - v(t) \big\|_{H^{s}} \big\|\bar{\eta}_{t}(t) \big\|_{H^{s}} \leq C M K_{2}(\delta) \big\|\bar{v}(t) - v(t)\big\|_{H^{s}}\\
& \qquad = C M K_{2}(\delta) \left\| \int_{0}^{t} \frac{\partial}{\partial \tau} (\bar{v}(\tau) - v(\tau)) d\tau \right\|_{H^{s}} \leq   C T M K_{2}(\delta) \big\|\bar{v}_{t}(t) - v_{t}(t)\big\|_{H^{s}} .
\end{split}
\end{equation*}
Therefore $\big(\tilde{G}_{v}(t) - \tilde{G}_{\bar{v}}(t)\big) \in C^{1}([0, T]; H^{s-2})$ which implies we can apply Proposition \ref{linear}. Hence from \eqref{eta-est-v-1} and \eqref{eta-est-v-2}, we obtain
\begin{equation*}
\begin{split}
&  \big\|\eta(t) - \bar{\eta}(t) \big\|_{H^{s}} \leq C \,T \int_{0}^{t} \Big( K_{2}(\delta) \big\|v(\tau) - \bar{v}(\tau) \big\|_{H^{s}} \\
& \quad \quad \quad + \big(K_{2}(\delta) \big\|v(\tau) - \bar{v}(\tau) \big\|_{H^{s}} +   \big\| g'(0)\big\|_{H^{s-1}} +1 \big) \big\|v_{\tau}(\tau)- \bar{v}_{\tau}(\tau)\big\|_{H^{s}} \Big) d\tau,  \\
 & \big\| \eta_{t}(t) - \bar{\eta}_{t}(t)\big\|_{H^{s}}   \leq C \int_{0}^{t} \Big( K_{2}(\delta) \big\|v(\tau) - \bar{v}(\tau) \big\|_{H^{s}} \\
 & \quad \quad \quad + \big(K_{2}(\delta)  \big\|v(\tau) - \bar{v}(\tau)\big\|_{H^{s}} +  \big\| g'(0)\big\|_{H^{s-1}} + 1\big) \big\|v_{\tau}(\tau)- \bar{v}_{\tau}(\tau)\big\|_{H^{s}}  \Big) d\tau.
\end{split}
\end{equation*}
For the first estimate we obtain
\begin{equation*}
\begin{split}
& \underset{t \in [0, T]}{\sup} \big\|\eta(t) - \bar{\eta}(t)\big\|_{H^{s}} \\
& \leq CT^{2} \Big(K_{2}(\delta) \underset{t \in [0, T]}{\sup} \big\|v(t) - \bar{v}(t)\big\|_{H^{s}} \\
& \quad  + \big(K_{2}(\delta)  \underset{t \in [0, T]}{\sup} \big\|v(t) - \bar{v}(t)\big\|_{H^{s}}+  \big\| g'(0)\big\|_{H^{s-1}} + 1\big)  \underset{t \in [0, T]}{\sup} \big\|v_{t}(t) - \bar{v}_{t}(t)\big\|_{H^{s}}\Big) \\
& = CT^{2}K_{2}(\delta) (1+ M) \underset{t \in [0, T]}{\sup} \big\|v(t) - \bar{v}(t)\big\|_{H^{s}} \\
& \qquad \qquad + C T^{2} \big(\big\| g'(0)\big\|_{H^{s-1}} + 1 \big) \underset{t \in [0, T]}{\sup} \big\|v_{t}(t) - \bar{v}_{t}(t)\big\|_{H^{s}}.
\end{split}
\end{equation*}
For the second one we have
\begin{equation*}
\begin{split}
& \underset{t \in [0, T]}{\sup} \big\| \eta_{t}(t) - \bar{\eta}_{t}(t) \big\|_{H^{s}}  \\
& \leq C T \Big(K_{2}(\delta) \underset{t \in [0, T]}{\sup} \big\|v(t) - \bar{v}(t) \big\|_{H^{s}} +  K_{2}(\delta) M \underset{t \in [0, T]}{\sup} \big\|v_{t}(t)- \bar{v}_{t}(t)\big\|_{H^{s}} \\
& \qquad \qquad +(\big\| g'(0)\big\|_{H^{s-1}} + 1)  \underset{t \in [0, T]}{\sup} \big\|v_{t}(t)- \bar{v}_{t}(t)\big\|_{H^{s}}\Big) \\
& = C T  K_{2}(\delta) \underset{t \in [0, T]}{\sup} \big\|v(t) - \bar{v}(t)\big\|_{H^{s}} \\
&\qquad \qquad + \Big(C T K_{2}(\delta) M + C T \big(\big\| g'(0)\big\|_{H^{s-1}} + 1\big)\Big)  \underset{t \in [0, T]}{\sup} \big\|v_{t}(t) - \bar{v}_{t}(t)\big\|_{H^{s}}.
\end{split}
\end{equation*}
This gives
\begin{equation*}
\begin{split}
& \!\!\!\!\!\!\!\!\!\!\!\!\!\! \underset{t \in [0, T]}{\sup} \big\|\eta(t) - \bar{\eta}(t)\big\|_{H^{s}}+ \underset{t \in [0, T]}{\sup} \big\| \eta_{t}(t) - \bar{\eta}_{t}(t)\big\|_{H^{s}} \\
&  \leq \Big(CT^{2}K_{2}(\delta)(1+  M) + C T K_{2}(\delta)\Big)\underset{t \in [0, T]}{\sup} \big\|v(t) - \bar{v}(t)\big\|_{H^{s}} \\
&  \qquad + \Big(C T^{2} \big(\big\| g'(0)\big\|_{H^{s-1}} + 1\big)+ C T K_{2}(\delta) M \\
& \qquad + C T \big(\big\| g'(0)\big\|_{H^{s-1}} + 1 \big)\Big)\underset{t \in [0, T]}{\sup} \big\|v_{t}(t) - \bar{v}_{t}(t)\big\|_{H^{s}}.
\end{split}
\end{equation*}
Therefore, we have 
\begin{equation*}
\begin{split}
& \!\!\!\!\!\!\!\!\!\!\!\!\!\!\!\!\!\! \big\|\mathcal{K}(v) - \mathcal{K}(\bar{v})\big\|_{X^{s}([0, T])}  = \underset{t \in [0, T]}{\sup} \Big(\big\|\eta(t) - \bar{\eta}(t)\big\|_{H^{s}} + \big\|\eta_{t}(t) - \bar{\eta}_{t}(t)\big\|_{H^{s}}\Big)  \\
& \leq C T K_{2}(\delta) \big(T (1+ M) + 1\big) \underset{t \in [0, T]}{\sup} \big\|v(t) - \bar{v}(t)\big\|_{H^{s}} \\
& \quad  + C T \big(T \big(\big\| g'(0)\big\|_{H^{s-1}} + 1 \big)+ K_{2}(\delta) M\\
&  \quad  + \big(\big\| g'(0)\big\|_{H^{s-1}} + 1 \big)\big) \underset{t \in [0, T]}{\sup}  \big\|v_{t}(t) - \bar{v}_{t}(t)\big\|_{H^{s}} \\
& \leq L  \underset{t \in [0, T]}{\sup} \big(\big\|v - \bar{v}\big\|_{H^{s}} +  \big\|v_{t} - \bar{v}_{t}\big\|_{H^{s}} \Big)\\
&  \leq L \big\|v - \bar{v}\big\|_{X^{s}([0, T])}, 
\end{split}
\end{equation*}
with $L = C T \max\big\{K_{2}(\delta)\big(T (1+ M) + 1\big), (1+ T) \big(\big\| g'(0)\big\|_{H^{s-1}} + 1 \big)+ K_{2}(\delta) M \big\}$. We can choose $T$ small enough so that $L < 1$. This implies that $\mathcal{K}$ is a contraction, so by Banach's Fixed Point Theorem we have existence of a local solution $\eta$ to \eqref{NL-eqn}-\eqref{ic-eta} and local well-posedness.
\end{proof}

\begin{remark}
Due to the method and techniques used, it is possible to obtain a similar local existence result for a bounded, open domain $\Omega \subset \mathbb{R}$ with suitable boundary conditions.
\end{remark}

\begin{remark}
Recalling that $H^{s} \subset C^{k}$ if $s > \frac{1}{2} + k$, we observe that the solutions are classical and by choosing $s$ large they can be made as smooth as we want.
\end{remark}

\begin{remark}
In the above proof a smallness condition is required on $\| \eta_{0}\|_{H^{s}}$ only, not on $\| \eta_{1}\|_{H^{s}}$.
\end{remark}

\noindent Recalling $\omega = \eta_{x}$ we can state Theorem \ref{local-existence} for $\omega$ as follows:

\begin{theorem}\label{local-existence-omega}
Let $s > 5/2$. Assume $g \in C^{r+1}$ with $r= [s]+1$. Assume also that $\omega_{0}, \omega_{1} \in H^{s-1}$ with $\| \omega_{0}\|_{H^{s}} \leq \frac{\bar{\delta}}{2 ( 1 + T_{0} K_{1}(\delta))}$ for some $T_{0} > 0$, where $\delta$ and $K_{1}(\delta)$ are as in Proposition \ref{linear}, and $\bar{\delta}$ is as in $X^{s}([0, T])$. Then, there exists a sufficiently small time $T > 0$ with $T \leq T_{0}$, such that the Cauchy problem  \eqref{nonlinear-eqn}-\eqref{icw} admits a unique solution $\omega \in X^{s-1}([0, T])$.
\end{theorem}

\section{A discussion about the original problem}\label{strain-lim}
\setcounter{equation}{0}

In the previous sections, under the assumption that the constitutive function $h(S)$ is strictly increasing, we have converted our original initial-value problem \eqref{PDE-T}-\eqref{ic} to the initial-value problem  \eqref{NL-eqn}-\eqref{ic-eta} and then we have proved existence of local solutions to this problem. In this section we first restate our existence result in terms of the displacement $u$ and the strain $u_{x}$. Then, we look at some examples to discuss the assumptions we made for the original problem \eqref{PDE-T}-\eqref{ic}. 

Recalling that $S= g(\omega)$ and $g$ is a sufficiently smooth function, the main conclusions in Theorem \ref{local-existence-omega} stated for $\omega$ are mainly true for the solution $S$ of the initial-value problem \eqref{PDE-T}-\eqref{ic}.

From the definitions $ \epsilon= u_{x}$, $\omega = \epsilon + \nu \epsilon_{t}$ and $\omega = \eta_{x}$, we actually have $\eta = u + \nu u_{t}$. Now, given $\eta$ we can rewrite this as
\[u_{t} + \frac{1}{\nu} u = \frac{1}{\nu} \eta,\]
which can be viewed as an ordinary differential equation in $u$. Solving it, we obtain
\begin{equation}\label{u}
u(x, t) = u(x, 0) e^{-\frac{t}{\nu}} + \frac{1}{\nu} \int_{0}^{t} e^{-\frac{1}{\nu}(t - \tau)} \eta(x, \tau) d\tau.
\end{equation}
Also, differentiating \eqref{u} we get 
\[u_{x}(x, t) = u_{x}(x, 0) e^{-\frac{t}{\nu}} + \frac{1}{\nu} \int_{0}^{t} e^{-\frac{1}{\nu}(t - \tau)} \omega(x, \tau) d\tau.\]
Since $\eta, \eta_{t} \in H^{s}$ and $\omega, \omega_{t} \in H^{s-1}$ for $s> 5/2,$ we immediately obtain $u, u_{t} \in H^{s}$ and $u_{x}, u_{xt} \in H^{s-1}$ for $s > 5/2$, respectively. As a result Theorem \ref{local-existence} and Theorem \ref{local-existence-omega} are valid for $u$ and $u_{x}$, respectively. 

Below we will consider some constitutive relations for which the invertibility condition is satisfied. However, several strain-limiting models for which the invertibility condition fails have been proposed in the literature.  For a more detailed discussion on the invertibility issue, we refer the reader to \cite{Mai-Walton} and \cite{Mai-Walton-2015} where  the notions of strong ellipticity and, as a weaker convexity notion, monotonicity are investigated for strain-limiting models in three-dimensional case. In \cite{Mai-Walton-2015}, Mai and Walton have  pointed out that, for the class of models studied therein,  monotonicity holds for strains with sufficiently small norms, and fails when strain is large enough. But at this point, we should note that, to stay within the context of strain-limiting theory, strain must remain small even for large stresses. In the present work we consider the following sample functions that are simplified versions of the nonlinear constitutive relations widely used in literature
\begin{eqnarray}
&& h(S)=\frac{S}{(1+ S^{2})^{1/2}},  \label{raj2011} \\
&& h(S)= \arctan S  \label{men2018},  \\
&&  h(S)=\left(-1+2\left(1+\frac{S^{2}}{2}\right)\right)S \label{kannan2014}
\end{eqnarray}
(where  the parameter values are chosen appropriately). We note that, in all cases, $h(0)=0$ as expected. The strain-limiting models \eqref{raj2011} and \eqref{kannan2014} correspond to the nonlinear constitutive relations proposed in \cite{Rajagopal-2011} and \cite{Kannan-Raj-Sac}, respectively. The model \eqref{men2018} has been recently proposed in \cite{Men-Orel-Bus}. The crucial fact about the above three cases is that we have  $h^{\prime}(S)>0$ for all $S$. This guarantees the existence of the inverse function $h^{-1}(\omega)=g(\omega)$ with $g(0) = 0$ on a suitable interval of $\omega$.

Our local existence proof for  \eqref{NL-eqn}-\eqref{ic-eta} is based on the assumption \eqref{pos}. For above functions, since $g'(\eta_{x}) > 0$ this is automatically satisfied. We now discuss how we can fix $\delta$ for the above particular forms of the constitutive relations. Assume that $\lim_{S\rightarrow \pm \infty}h(S)=\alpha_{\pm}$. In such case the inverse of $h$, that is $g(w)$, exists for $\alpha_{-}<\omega <\alpha_{+}$. This implies that we can take any $\delta$ satisfying  $0<\delta\leq \min\{\vert\alpha_{-}\vert,\vert\alpha_{+}\vert\}$. For the models \eqref{raj2011},  \eqref{men2018} and \eqref{kannan2014} we have $\alpha_{\pm}=\pm 1$, and $\alpha_{\pm}=\pm \frac{\pi}{2}$ and $\alpha_{\pm}=\pm \infty$, respectively. Consequently, for the three cases, the restriction for $\delta$ becomes $0<\delta <1$, $0<\delta < \frac{\pi}{2}$ and $0<\delta <\infty$, respectively.

\appendix

\section{Appendix} \label{Appendix}

The analysis in Section \ref{loc-existence} is based on the assumptions $\eta_{0}(\pm \infty) =0$ and $\eta_{1}(\pm \infty) = 0$. Since $\omega_{0}(x)= (\eta_{0}(x))_{x}$ and $\omega_{1}(x) = (\eta_{1}(x))_{x}$ these assumptions imply 
\begin{equation}\label{int-w-0}
\int_{-\infty}^{\infty} \omega_{0}(x) dx = 0 \quad \text{and} \quad \int_{- \infty}^{\infty} \omega_{1}(x) dx = 0,
\end{equation}
respectively. However, conditions \eqref{int-w-0} on the initial data $\omega_{0}$ and $\omega_{1}$ of the initial-value problem \eqref{nonlinear-eqn}-\eqref{icw} give further restrictions on $\omega$ and $\omega_{t}$ valid for any interval of time where the solution exists. This can be explained as follows. Integrating \eqref{omega-phi-1} in $x$ on $\mathbb{R}$ and noting that $\phi \to 0$ ($\eta_{t} \to 0$) as $x \to \pm \infty$, we get the conserved quantity 
\[\int_{-\infty}^{\infty} \omega(x, t) dx = \int_{-\infty}^{\infty} \omega(x, 0) dx.\]
Using the first condition in \eqref{int-w-0} with $\omega_{0}(x) = \omega(x, 0)$ we then obtain
\begin{equation}\label{int-w}
\int_{-\infty}^{\infty} \omega(x, t) dx =0
\end{equation}
for all $t$. Similarly, rewriting \eqref{nonlinear-eqn} in the form $ (\omega_{t})_{t} = \big(g(\omega)_{x} + \nu g(\omega)_{xt}\big)_{x}$, integrating in $x$ on $\mathbb{R}$ and noting that the terms on the right tend to zero as $x \to \pm\infty$, we get another conserved quantity
\[\int_{-\infty}^{\infty} \omega_{t}(x, t) dx= \int_{-\infty}^{\infty} \omega_{t}(x, 0) dx.\]
On the other hand, using the second condition in \eqref{int-w-0} with $\omega_{1} = \omega_{t}(x, 0)$ in this equation we obtain
\begin{equation}\label{int-wt}
\int_{-\infty}^{\infty} \omega_{t}(x, t)dx = 0
\end{equation}
for all $t$. The role of conditions \eqref{int-w} and \eqref{int-wt} is to guarantee that the initial data $\eta_{0}$ and $\eta_{1}$ satisfy $\eta_{0}(\pm \infty) =0$ and $\eta_{1}(\pm \infty) = 0$.
Integrating \eqref{omega-phi-2} in $x$ over $\mathbb{R}$ and using the zero conditions at infinity $\sigma(\pm \infty, t) = 0$, we obtain another conserved quantity 
\[\int_{-\infty}^{\infty} \phi(x, t) dx = \int_{-\infty}^{\infty} \phi(x, 0) dx,\]
for all $t$. 

\section*{Acknowledgements}

This  work  was  supported  by  the  Scientific  and  Technological Research Council of Turkey (T\"{U}B\.{I}TAK) under the grant 116F093.

\section*{References}
\bibliographystyle{plainnat}
\bibliography{bibliography}

\begin{thebibliography}{26}
\providecommand{\natexlab}[1]{#1}
\providecommand{\url}[1]{\texttt{#1}}
\expandafter\ifx\csname urlstyle\endcsname\relax
  \providecommand{\doi}[1]{doi: #1}\else
  \providecommand{\doi}{doi: \begingroup \urlstyle{rm}\Url}\fi

\bibitem[B\'arta(2014)]{Barta}
T.~B\'arta.
\newblock One-dimensional model describing the non-linear viscoelastic response
  of materials.
\newblock \emph{Comment. Math. Univ. Carolin.}, 55\penalty0 (2):\penalty0
  227--246, 2014.

\bibitem[Bul\'{\i}\v{c}ek et~al.(2012)Bul\'{\i}\v{c}ek, M\'{a}lek, and
  Rajagopal]{BuMaRaj}
M.~Bul\'{\i}\v{c}ek, J.~M\'{a}lek, and K.~R. Rajagopal.
\newblock On {K}elvin-{V}oigt model and its generalizations.
\newblock \emph{Evol. Equ. Control Theory}, 1\penalty0 (1):\penalty0 17--42,
  2012.

\bibitem[Bul\'{\i}\v{c}ek et~al.(2014)Bul\'{\i}\v{c}ek, M\'{a}lek, Rajagopal,
  and S\"uli]{BuMaRaSu}
M.~Bul\'{\i}\v{c}ek, J.~M\'{a}lek, K.~R. Rajagopal, and E.~S\"uli.
\newblock On elastic solids with limiting small strain: Modelling and analysis.
\newblock \emph{EMS Surv. Math. Sci.}, 1\penalty0 (2):\penalty0 283--332, 2014.

\bibitem[Bustamante(2009)]{Bus-2009}
R.~Bustamante.
\newblock Some topics on a new class of elastic bodies.
\newblock \emph{Proc. R. Soc. A}, 465:\penalty0 1377--1392, 2009.

\bibitem[Bustamante and Rajagopal(2011)]{Bus-Raj-2011}
R.~Bustamante and K.~R. Rajagopal.
\newblock Solutions of some simple boundary value problems within the context
  of a new class of elastic materials.
\newblock \emph{Int. J. Nonlinear Mech.}, 46\penalty0 (2):\penalty0 376--386,
  2011.

\bibitem[Constantin and Molinet(2002)]{Cons-Mol}
A.~Constantin and L.~Molinet.
\newblock The initial-value problem for a generalized \text{B}oussinesq
  equation.
\newblock \emph{Differ. Integral Equ.}, 15\penalty0 (9):\penalty0 1061--1072,
  2002.

\bibitem[Erbay and \c{S}eng\"ul(2015)]{Erbay-Sengul}
H.~A. Erbay and Y.~\c{S}eng\"ul.
\newblock Traveling waves in one-dimensional nonlinear models of
  strain-limiting viscoelasticity.
\newblock \emph{Int. J. Nonlinear Mech.}, 77:\penalty0 61--68, 2015.

\bibitem[Hille(1948)]{Hille}
E.~Hille.
\newblock \emph{Functional Analysis and Semi-Groups}.
\newblock American Mathematical Society, 1948.

\bibitem[Huang et~al.(2017)Huang, Rajagopal, and Dai]{Huang-Raj-Dai}
S.~J. Huang, K.~R. Rajagopal, and H.~H. Dai.
\newblock Wave patterns in a nonclassic nonlinearly-elastic bar under
  \text{R}iemann data.
\newblock \emph{Int. J. Nonlinear Mech.}, 91:\penalty0 76--85, 2017.

\bibitem[Kannan et~al.(2014)Kannan, Rajagopal, and Saccomandi]{Kannan-Raj-Sac}
K.~Kannan, K.~R. Rajagopal, and G.~Saccomandi.
\newblock Unsteady motions of a new class of elastic solids.
\newblock \emph{Wave Motion}, 51:\penalty0 833--843, 2014.

\bibitem[Kawashima and Shibata(1992)]{Kaw-Shi}
S.~Kawashima and Y.~Shibata.
\newblock Global existence and exponential stability of small solutions to
  nonlinear viscoelasticity.
\newblock \emph{Commun. Math. Phys.}, 148:\penalty0 189--208, 1992.

\bibitem[Kobayashi et~al.(1993)Kobayashi, Pecher, and Shibata]{Kob-Pec-Shi}
T.~Kobayashi, H.~Pecher, and Y.~Shibata.
\newblock On a global time existence theorem of smooth solutions to a nonlinear
  wave equation with viscosity.
\newblock \emph{Math. Ann.}, 296:\penalty0 215--234, 1993.

\bibitem[Mai and Walton(2014)]{Mai-Walton}
T.~Mai and J.~R. Walton.
\newblock On strong ellipticity for implicit and strain-limiting theories of
  elasticity.
\newblock \emph{Math. Mech. Solids}, 20\penalty0 (2):\penalty0 121--139, 2014.

\bibitem[Mai and Walton(2015)]{Mai-Walton-2015}
T.~Mai and J.~R. Walton.
\newblock On monotonicity for strain-limiting theories of elasticity.
\newblock \emph{J. Elast.}, 120:\penalty0 39--65, 2015.

\bibitem[Meneses et~al.(2018)Meneses, Orellana, and Bustamante]{Men-Orel-Bus}
R.~Meneses, O.~Orellana, and R.~Bustamante.
\newblock A note on the wave equation for a class of constitutive relations for
  nonlinear elastic bodies that are not \text{G}reen elastic.
\newblock \emph{Math. Mech. Solids}, 23\penalty0 (2):\penalty0 148--158, 2018.

\bibitem[Muliana et~al.(2013)Muliana, Rajagopal, and Wineman]{Mul-Raj-Wine}
A.~Muliana, K.~R. Rajagopal, and A.~S. Wineman.
\newblock A new class of quasi-linear models for describing the nonlinear
  viscoelastic response of materials.
\newblock \emph{Acta Mech}, 224:\penalty0 2169--2183, 2013.

\bibitem[Pazy(1983)]{Pazy}
A.~Pazy.
\newblock \emph{Semigroups of Linear Operators and Applications to Partial
  Differential Equations}.
\newblock Springer, 1983.

\bibitem[Pecher(1980)]{Pecher}
H.~Pecher.
\newblock On global regular solutions of third order partial differential
  equations.
\newblock \emph{J. Math. Anal. Appl.}, 73:\penalty0 278--299, 1980.

\bibitem[Rajagopal(2003)]{Rajagopal-2003}
K.~R. Rajagopal.
\newblock On implicit constitutive theories.
\newblock \emph{Appl. Math.}, 48:\penalty0 279--319, 2003.

\bibitem[Rajagopal(2007)]{Rajagopal-2007}
K.~R. Rajagopal.
\newblock The elasticity of elasticity.
\newblock \emph{Z. Angew. Math. Phys.}, 58:\penalty0 309--317, 2007.

\bibitem[Rajagopal(2009)]{Rajagopal-2009}
K.~R. Rajagopal.
\newblock A note on a reappraisal and generalization of the {K}elvin-{V}oigt
  model.
\newblock \emph{Mech. Res. Commun.}, 36:\penalty0 232--235, 2009.

\bibitem[Rajagopal(2011)]{Rajagopal-2011}
K.~R. Rajagopal.
\newblock Non-linear elastic bodies exhibiting limiting small strain.
\newblock \emph{Math. Mech. Solids}, 16:\penalty0 122--139, 2011.

\bibitem[Rajagopal(2014)]{Rajagopal-2014}
K.~R. Rajagopal.
\newblock On the nonlinear elastic response of bodies in the small strain
  range.
\newblock \emph{Acta. Mech.}, 225:\penalty0 1545--1553, 2014.

\bibitem[Rajagopal and Saccomandi(2014)]{Raj-Sac-2014}
K.~R. Rajagopal and G.~Saccomandi.
\newblock Circularly polarized wave propagation in a class of bodies defined by
  a new class of implicit constitutive relations.
\newblock \emph{Z. Angew. Math. Phys.}, 65:\penalty0 1003--1010, 2014.

\bibitem[Saito et~al.(2003)Saito, Furuta, Hwang, Kuramoto, Nishino, Suzuki,
  Chen, Yamada, Ito, Seno, Nonaka, Ikehata, Nagasako, Iwamoto, Ikuhara, and
  Sakuma]{Saito-et-all}
T.~Saito, T.~Furuta, J.~Hwang, S.~Kuramoto, K.~Nishino, N.~Suzuki, R.~Chen,
  A.~Yamada, K.~Ito, Y.~Seno, T.~Nonaka, H.~Ikehata, N.~Nagasako, C.~Iwamoto,
  Y.~Ikuhara, and T.~Sakuma.
\newblock Multifunctional alloys obtained via a dislocation-free plastic
  deformation mechanism.
\newblock \emph{Science}, 300:\penalty0 464--467, 2003.

\bibitem[Wloka(1987)]{Wloka}
J.~Wloka.
\newblock \emph{Partial Differential Equations}.
\newblock Cambridge University Press, 1987.

\end{thebibliography}

\end{document}